\documentclass[11pt]{amsart}

\usepackage{tikz}
\usepackage{amssymb}
\usepackage{float}

\newtheorem{theorem}{Theorem}[section]
\newtheorem{corollary}[theorem]{Corollary}
\newtheorem{lemma}[theorem]{Lemma}

\newtheorem{conjecture}[theorem]{Conjecture}

\newtheorem{maintheorem}{Theorem}

\theoremstyle{definition}

\theoremstyle{remark}

\newtheorem{case}{Case}

\numberwithin{equation}{section}
\numberwithin{figure}{section}

\newcommand{\Z}{\mathbb Z}

\newcommand{\eps}{\varepsilon}

\newcommand{\U}{\mathcal U}
\newcommand{\cR}{\mathcal R}
\newcommand{\diff}{\operatorname{Diff}}
\newcommand{\phc}{\operatorname{Phc}}

\newcommand{\transv}{\pitchfork}
\newcommand{\zeroeq}{\stackrel{\circ}{=}}

\DeclareMathOperator{\nuh}{Nuh}
\DeclareMathOperator{\pper}{Per}

\newcommand{\eval}[2][\right]{\relax
  \ifx#1\right\relax \left.\fi#2#1\rvert}

\begin{document}
\title[Minimality and stable Bernoulliness]{Minimality and stable Bernoulliness in dimension 3}

\author[Gabriel N\'u\~nez]{Gabriel N\'u\~nez}
\address{Gabriel N\'u\~nez, IMERL, Facultad de Ingenier\'ia,
Universidad de la Rep\'ublica,
Julio Herrera y Reissig 565,
11.300 Montevideo, Uruguay. \newline
Departamento de matem\'atica, 
Facultad de Ingenier\'ia y Tecnolog\'ias,
Universidad Cat\'olica del Uruguay, 
Comandante Braga 2715, 
11.600 Montevideo, Uruguay.
}
\email{fnunez@fing.edu.uy}
\author[Jana Rodriguez Hertz]{Jana Rodriguez Hertz}
\address{Jana Rodriguez Hertz,
Department of Mathematics, 
Southern University of Science and Technology of China.
No 1088, xueyuan Rd., Xili, Nanshan District, Shenzhen,Guangdong, China 518055.}
\email{rhertz@sustc.edu.cn}

\begin{thanks}{GN  was supported by Agencia Nacional de Investigaci\'on e Innovaci\'on. The research that gives rise to the results presented in this publication received funds from the Agencia Nacional de Investigaci\'on e Innovaci\'on under the code POS\_NAC\_2014\_1\_102348}
\end{thanks}
\begin{thanks}
{JRH was supported by SUSTech and NSFC 11871262} 
\end{thanks}

\begin{abstract}
In 3-dimensional manifolds, we prove that generically in $\diff^{1}_{m}(M^{3})$, the existence of a minimal expanding invariant foliation implies stable Bernoulliness.
\end{abstract}
\maketitle

\section{Introduction}
\subsection{Historical context}
Let $M$ be a smooth compact manifold and let $m$ be a smooth volume measure. A diffeomorphism $f: M \rightarrow M$ is {\em ergodic} if the Birkhoff's limits
$$\tilde{\varphi}(x)=\lim_{n\rightarrow\infty}\frac{1}{n}\sum\limits_{j=0}^{n-1} \varphi \circ f^{j}(x)$$
are constant for $m$-almost every $x \in M$ for all $\varphi:M \rightarrow \mathbb{R}$ continuous function.\par
%
In 1939 E. Hopf \cite{Hopf1939} proved that the geodesic flow of a surface with negative sectional curvature is ergodic with respect to the Liouville measure. This result was generalized by D. Anosov \cite{AN1969}, who in the late sixties showed that conservative $C^2$-Anosov diffeomorphisms (and flows) are $C^1$ {\em stably ergodic}. A $C^2$ conservative diffeomorphism is stably ergodic if there exists a $C^1$ neighborhood $\mathcal{U}\subset\diff^{1}_{m}(M)$ such that every conservative $C^2$ element in $\mathcal{U}$ is ergodic. Similarly for the case of flows. A crucial tool in the proof is the absolute continuity of the stable and unstable foliations, due to D. Anosov and Ya. Sinai \cite{AS1967}. It is not known yet whether $C^1$-Anosov diffeomorphisms are necessarily ergodic.\par
Until 1993, Anosov diffeomorphisms were the only known conservative examples of stably ergodic diffeomorphisms, but Grayson, Pugh and Shub showed that the time-one map of the geodesic flow of a surface of constant negative curvature is stably ergodic \cite{GPS1994}. This result was generalized by Wilkinson, for variable negative curvature in \cite{wilkinson1998}.\par
It is natural to look for mechanisms that activate stable ergodicity for a smooth volume measure.  In 1995, Pugh and Shub proposed that ``a little hyperbolicity goes a long way toward guaranteeing stable ergodicity'', and they conjectured that {\em partially hyperbolic diffeomorphisms} are $C^{r}$-generically stably ergodic. The Pugh-Shub conjecture was stated in the International Congress on Dynamical Systems, held in Montevideo in 1995, in the memory of Ricardo Ma\~n\'e \cite{PS1995}. It has become a very active topic of research.\par
A partially hyperbolic diffeomorphism satisfies that the tangent bundle $TM$ splits into three $Df$-invariant subbundles $E^u \oplus E^c \oplus E^s$, where $Df$ is expanding on $E^u$ (called the unstable bundle), contracting on $E^s$ (the stable bundle) and  has an intermediate behavior on $E^c$ (the central bundle), that is, vectors in $E^{c}$ may be expanded or contracted by $Df$, but not as strongly as in $E^{u}$ or $E^{s}$, respectively.  See Section \ref{section.definitions} for a precise definition.\par

The Pugh-Shub conjecture was proved for $r=\infty$ by F. Rodriguez Hertz, J. Rodriguez Hertz and R. Ures \cite{HHU2008} when the central subbundle is one dimensional. For $r=1$ and any dimension of central subbundle, it was proven by A. Avila, S. Crovisier y A. Wilkinson \cite{ACW2017}. The conjecture remains open for $r=2$, that is, it is not known yet whether stably ergodic diffeomorphisms are $C^{2}$-dense among volume preserving diffeomorphisms.\par

Since Pugh and Shub proposed ``a little hyperbolicity'' as a mechanism activating stable ergodicity, a natural question is how far we could go. Is partial hyperbolicity the least we can ask? Does stable ergodicity imply partial hyperbolicity? The answer is no. A. Tahzibi  gave an example of a stably ergodic diffeomorphims which is not partially hyperbolic  in his Ph.D. Thesis \cite{T2004}. The map, a diffeomorphism of $T^4$, was introduced before by Bonatti-Viana in \cite{BV2000}. Even though the map is not partially hyperbolic it has a {\em dominated splitting}, that is: the tangent bundle over $M$ splits into two $Df$-invariant subbundles $TM= E \oplus F$ such that given any $x\in M$, any unitary vectors $v_E \in E(x)$ and $v_F \in F(x)$:
$$\parallel Df^{N}(x) (v_E) \parallel \leq \frac{1}{2}\parallel Df^{N}(x) (v_F) \parallel$$
for some $N>0$ independent of $x$.\newline\par 
Is this the least we can ask to obtain stable ergodicity? The following conjecture suggests that generically in $\diff^{1}_{m}(M)$, it is. 
Remember that, after Theorem \ref{teo.jana.acw}, having a dominated splitting is generically equivalent to having positive metric entropy with respect to Lebesgue measure:
\begin{conjecture}\label{conjecture.stable.ergodicity}
 Generically in $\diff^{1}_{m}(M)$, if $f$ has positive metric entropy with respect to Lebesgue measure, then $f$ is stably ergodic. 
\end{conjecture}
If this conjecture were true, then Pugh and Shub statement would be true: generically either all Lyapunov exponents are zero $m$-almost everywhere, or else you have stable ergodicity. \par
For many years, we addressed this conjecture in the context of 3-manifolds. This context has a clear advantage: when there is a dominated splitting at least one of the invariant subbundles has to be hyperbolic due to the fact that $f$ preserves volume. This problem proved more difficult than originally thought. However, we found a condition that guarantees not only stable ergodicity but {\em stable Bernoulliness}. A diffeomorphism $f\in\diff^{1}_{m}(M)$ is {\em stably Bernoulli} if there exists a $C^{1}$-neighborhood $\U\subset\diff^{1}_{m}(M)$ of $f$ such that all $g\in\U\cap \diff^{2}_{m}(M)$ are Bernoulli, that is, are metrically isomorphic to a Bernoulli shift. With this definition a $C^{1}$ stably Bernoulli diffeomorphism might not be Bernoulli. We are aware that this is not a standard definition, but for practical reasons we will state it like this. Of course, a $C^{2}$ stably Bernoulli diffeomorphism is Bernoulli. \par
The mechanism we propose as a generic stable ergodicity activator is the minimality of an expanding or contracting invariant foliation. 
A foliation $W$ is {\em minimal} if every leaf $W(x)$ of $W$ is dense in $M$. An $f$-invariant foliation $W$ is {\em contracting} if $\|Df(x)|_{TW}\|<1$ for every $x\in M$. An $f$-invariant foliation is {\em expanding} if it is contracting for $f^{-1}$.\newline\par
\begin{maintheorem}{\label{teo.stable.bernoulli}} There exists a residual set $\cR$ in $\diff^{1}_{m}(M^3)$ such that for $f\in \cR$, if there exists a minimal expanding or contracting $f$-invariant foliation, then  $f$ is stably Bernoulli. 
\end{maintheorem}
\vspace*{1em}\par
This condition might not strike in principle as ``a little hyperbolicity'', but in fact we do not know how frequently this condition is found among the volume preserving diffeomorphisms of 3-manifolds with a dominated splitting. We conjecture the following:
\begin{conjecture}\label{conjecture.dim3}
Generically in $\diff^{1}_{m}(M^{3})$, either all Lyapunov exponents are zero almost everywhere, or else there exists a minimal invariant expanding or contracting foliation.  
\end{conjecture}
Observe that in Theorem \ref{teo.stable.bernoulli}, even though we cannot guarantee that $f\in\cR$ is Bernoulli, it will always be ergodic, since ergodicity is a $G_{\delta}$-property. \\

\subsection{Related results} The proof of Theorem \ref{teo.stable.bernoulli} relies on the SB-criterion (Theorem \ref{teo.no.generico}). This criterion provides conditions to obtain stable Bernoulliness. It is not a genericity-type theorem. All diffeomorphisms satisfying the hypothesis are stably Bernoulli. \par
D. Obata \cite[Theorem B]{obata2019} recently introduced a criterion for stable Bernoulliness in the line of Theorem \ref{teo.no.generico}. Even though both theorems present some coincidences, in the sense that they do not require partial hyperbolicity, they have complementary hypotheses. We would like to explain their differences and similarities. Similarities: both theorems require a $Df$-expanding bundle $E^{uu}$ dominating another bundle $F$, that is $TM=F\oplus_{<}E^{uu}$. Also, both require a certain positive measure set of points for which all Lyapunov exponents corresponding to vectors in $F$ are negative. Differences: D. Obata's theorem requires chain hyperbolicity. This is a type of topological hyperbolicity that contracts topologically in the direction of $F$. Not only that, it requires a sort of integrability: there has to exist a plaque family $W^{F}_{x}$ tangent to $F$, so that $f(\overline{W^{F}_{x}})\subset W^{F}_{f(x)}$ for all $x\in M$. It also requires the existence of a hyperbolic periodic point $p$ such that its homoclinic class is all $M$. We do not require integrability but, on the other hand, we ask the positive measure set of negative Lyapunov exponents to exist on a $C^{1}$-neighborhood of $f$ for all $C^{2}$ conservative $g$. And we ask the foliation tangent to $E^{uu}$ to be minimal for the diffeomorphism $f$.\par
We also point out that there is another theorem in \cite{obata2019}, which does not require the existence of an expanding bundle. It requires a dominated bundle $TM=F\oplus_{<}E$ where 
$$\int_{M}\log\|Df|_{F}\|dm<0\qquad\text{and}\qquad\int_{M}\log\|Df^{-1}|_{E}\|dm<0$$
but he still requires chain hyperbolicity (for both bundles). That is, the existence of an invariant family of plaques not only for $F$ but also for $E$ (for $f^{-1}$), and the presence of a hyperbolic periodic point whose homoclinic class is all $M$. It is also required the existence of two hyperbolic periodic points $q_{s}$ and $q_{u}$ homoclinically related to $p$ (see definitions in \S\ref{section.definitions}) such that $W^{s}(q_{s})$ contains the plaque $W^{F}_{q_{s}}$ and $W^{u}(q_{u})$ contains the plaque $W^{E}_{q_{u}}$. Notice that this hypothesis is also required in \cite[Theorem B]{obata2019} mentioned above.\par
So, as we said, they can be regarded as complementary results. 
\section{Basic concepts}\label{section.definitions}
Let $f\in\diff^{1}_{m}(M)$ be a volume preserving diffeomorphism. Given a $Df$-invariant splitting $TM=E\oplus F$, we say that the splitting is {\em dominated}, and denote it by $TM=E\oplus_{<} F$ if there exists a Riemannian metric such that for all unit vectors $v_{E}\in E_{x}$ and $v_{F}\in F_{x}$ we have
$$\|Df(x)v_{E}\| \leq \frac{1}{2}\|Df(x)v_{F}\|$$
We say that the $Df$-invariant splitting $TM=E_{1}\oplus \dots\oplus E_{k}$ is dominated if $E_{1}\oplus\dots E_{j}\oplus_{<}E_{j+1}\oplus\dots\oplus E_{k}$ for all $j=1,\dots,k-1$.\par
We say that $f\in\diff^{1}_{m}(M)$ is {\em partially hyperbolic} if $TM$ admits a dominated splitting of the form 
$$TM=E^{s}\oplus E^{c}\oplus E^{u}$$
where $Df$ is contracting on $E^{s}$ and $Df^{-1}$ is contracting on $E^{u}$. That is, if for all unit vectors
$v^{\sigma}\in E^{\sigma}_{x}$, where $\sigma=s,c,u$, we have
$$\|Df(x)v^{s}\|<1<\|Df(x)v^{u}\|\quad\text{and}\quad\|Df(x)v^{s}\|<\|Df(x)v^{c}\|<\|Df(x)v^{u}\|.$$

Let us recall that $\lambda(x,v)$ is the {\em Lyapunov exponent}  associated to $v\in T_{x}M$ if
$$\lambda(x,v)=\overline{\lim_{n\to\infty}}\frac{1}{n}\log\|Df^{n}(x)v\|$$
For $m$-almost every $x\in M$, there are finitely many Lyapunov exponents $\lambda_{1}(x),\dots,\lambda_{k}(x)$ in $T_{x}M$, and there is a measurable $Df$-invariant splitting, called the {\em Oseledets splitting} 
$$T_{x}M=E^{1}_{x}\oplus\dots\oplus E^{k}_{x}$$
such that $\lambda(x,v_{i})=\lambda_{i}(x)$ for all $v_{i}\in T_{x}M\setminus\{0\}$. See for instance, \cite{Pesin1977}. 
We will denote by $\nuh(f)$ the set of $x$ such that all $\lambda(x,v)$ are different from zero. 
The measure $m$ is called {\em hyperbolic} if $\nuh(f)$ has full $m$-measure.\par
For simplicity, we will denote
\begin{equation}\label{zipped.oseledets.splitting}
T_{x}M=E^{-}_{x}\oplus E^{0}_{x}\oplus E^{+}_{x} 
\end{equation}

the splitting such that $\lambda(x,v^{+})>0$ for all $v^{+}\in E^{+}_{x}$, $\lambda(x,v^{0})=0$ for all $v^{0}\in E^{0}_{x}$ and $\lambda(x,v^{-})<0$ for all $v^{-}\in E^{-}_{x}$. We call this splitting the {\em zipped Oseledets splitting}.
For $x\in M$, let us define
\begin{equation}\label{eq.invariant.manifolds}
 W^{\pm}(x)=\{y\in M: \limsup_{n\to\infty}\frac{1}{n}\log d(f^{\mp n}(x),f^{\mp n}(y))<0\}
\end{equation}
If $f\in\diff^{2}_{m}(M)$ then for $m$-almost every point $x$, $W^{+}(x)$ and $W^{-}(x)$ are smooth immersed manifolds \cite{Pesin1977}. For $f\in\diff^{1}_{m}(M)$ this is not necessarily true \cite{Pugh1984}.
However, if the zipped Oseledets splitting is dominated, then both $W^{+}(x)$ and $W^{-}(x)$ are immersed manifolds for $m$-almost every point, see \cite{ABC2011}.\par

Following \cite{HHTU11}, given a hyperbolic periodic point $p\in M$ we define the {\em stable Pesin homoclinic class} of $p$ by 
\begin{equation}\label{eq.phc.-}
\phc^{-}(p)=\{x: W^{-}(x)\transv W^{u}(o(p))\ne\emptyset\} 
\end{equation}
where $W^{u}(o(p))$ is the union of the unstable manifolds of $f^{k}(p)$, for all $k=0,\dots,{\rm per}(p)-1$. $\phc^{-}(p)$ is invariant and saturated by $W^{-}$-leaves. We will also denote by $o(p)$ the orbit of $p$. Analogously, we define 
\begin{equation}\label{eq.phc.+}
\phc^{+}(p)=\{x: W^{+}(x)\transv W^{s}(o(p))\ne\emptyset\} 
\end{equation}

If there exists an expanding foliation $W^{u}$, we will denote 
\begin{equation}\label{eq.phc.u}
 \phc^{u}(p)=\{x\in M: W^{u}(x)\transv W^{s}(o(p))\ne\emptyset\}
\end{equation}
Analogously we define $\phc^{s}(p)$ if a contracting foliation $W^{s}$ is given. The foliation will be clear from the context, if it is not, we will denote these sets by $\phc^{W}(p)$, where $W$ is given. 

The importance of Pesin homoclinic classes comes from the next criterion of ergodicity:

\begin{theorem}[Theorem A, \cite{HHTU11}]\label{criterion_hhtu} Let $f:M \rightarrow M$ be a $C^2$-diffeomorphism over a closed connected Riemannian manifold $M$, let $m$ be a smooth invariant measure and $p \in \pper_H(f)$. If $m(\phc^{+}(p))>0$ and $m(\phc^{-}(p))>0$, then
\begin{enumerate}
	\item $\phc^{+}(p) \circeq \phc^{-}(p) \circeq \phc(p)$.
	\item $m|\phc(p)$ is ergodic.
	\item $\phc(p) \subset \nuh(f)$.
\end{enumerate}
\end{theorem}
The notation $A\zeroeq B$ means $m(A\triangle B)=0$.

As a consequence of the Pesin's ergodic component theorem \cite{Pesin1977}, Katok's closing lemma \cite{Katok1980} and the ergodicity criterion statement above, one obtains the following ``spectral decomposition'':

\begin{theorem} Let $M$ be a closed connected Riemannian manifold, let $f:M \rightarrow M$ be a $C^2$-diffeomorphism and let $m$ be a smooth measure hyperbolic over an $f$-invariant set $V$. Then:
\begin{enumerate}
	\item[(a)] 
	$$V \circeq \bigcup_{n \in \mathbb{N}} \Gamma_n$$
	where $\Gamma_n$ are disjoint measurable invariant sets such that $f|_{\Gamma_n}$ is ergodic.
	\item[(b)] For each $\Gamma_n$, there exists $k_n \in \mathbb{N}$ and measurable sets with positive measure $\Gamma_1^n, \, \Gamma_2^n, \ldots , \, \Gamma_{k_n}^n$ which are pairwise disjoints such that $f(\Gamma_j^n)=\Gamma_{j+1}^n$ for every $j=1, 2, \ldots, k_n-1$,  $f(\Gamma_{k_n}^n)=\Gamma_{1}^n$ and $f^{k_n}$ is Bernoulli.
	\item[(c)] There exists a hyperbolic periodic point $p_n$ such that $\Gamma_n= \phc (p_n)$
\end{enumerate}
\end{theorem}

\begin{corollary}\label{corollary.pesin} In the hypothesis of the theorem above, if $m$ is hyperbolic and $f^{n}$ is ergodic for every $n$, then $f$ is Bernoulli. 
\end{corollary}

\begin{corollary}\label{corollary.homoclinically.related} If for $f\in\diff^{2}_{m}(M)$ there exists a hyperbolic periodic point such that $m(\phc(p))=1$ and all points in the orbit of $p$ are homoclinically related, that is $W^{s}(f^{k}(p))\transv W^{u}(f^{m}(p))\ne\emptyset$ for all $k,m\in\Z$, then $f$ is Bernoulli. 
\end{corollary}
\begin{proof} Let us first note that if $p$ and $q$ are two hyperbolic periodic points that are homoclinically related, that is, $W^{s}(p)\transv W^{u}(q)\ne\emptyset$ and $W^{s}(q)\transv W^{u}(p)\ne\emptyset$, then $\phc(p)\zeroeq\phc(q)$. This is a consequence of the $\Gamma$-lemma. \par
Secondly, note that $\phc(p)$ depends on the iterate of $f$, that is, we actually have to consider $\phc_{f^{k}}(p)$. This is because the orbit of $p$ depends on the iterate, we actually have $o_{f^{k}}(p)\subset o_{f}(p)$. Therefore, in general $\phc_{f^{k}}(p)\subset\phc(p)$. However, by the previous remark, $\phc_{f^{k}}(p)\zeroeq\phc(p)\zeroeq M$ for all $k\ne 0$. By Corollary \label{corollary.pesin} above and Theorem \ref{criterion_hhtu}, $f$ is Bernoulli.
\end{proof}
The condition in the corollary above is generic in $\diff^{1}_{m}(M)$. In fact, the following holds (We thank S. Crovisier and R. Potrie for pointing us out the correct citation of this Theorem.):
\begin{theorem}\cite{AC2012}\label{teo.bonatti.crovisier} Generically in $\diff^{1}_{m}(M)$, all periodic points of the same index are homoclinically related.
\end{theorem}

The following theorem is essential in this paper:
\begin{theorem}\cite{mane1984,bochi2002,JRH2012,ACW16}\label{teo.jana.acw} For a generic $f\in\diff^{1}_{m}(M)$, either all Lyapunov exponents are zero $m$-almost everywhere, or else:
\begin{enumerate}
 \item $\nuh(f)\zeroeq M$
 \item $f$ is ergodic
 \item the Oseledets splitting is dominated. Call the zipped Oseledets splitting $TM=E^{-}\oplus E^{+}$
\item there exists a hyperbolic periodic point $q$ with $u(q)=\dim E^{+}$ such that $\phc(q)\zeroeq M$
\end{enumerate} 
\end{theorem}
It was R. Ma\~n\'e the first to notice that the Oseledets splitting had a regular behavior under generic conditions. He conjectured that generically the Oseledets splitting was dominated and announced that for the generic conservative diffeomorphism in dimension 2, either all Lyapunov exponents vanish or else the diffeomorphism is Anosov. He provided a sketch of the proof. The proof was harder than what was initially thought, and was finished by J. Bochi in 2002. Ten years later J. Rodriguez Hertz proved that a weaker result, namely, what is stated in Theorem \ref{teo.jana.acw} is true for 3-dimensional manifolds. Very soon after, A. Avila and J.Bochi conjectured that J. Rodriguez Hertz' theorem actually holds in any dimensions \cite{AB2012}. This conjecture was finally proven true by A. Avila, S. Crovisier and A. Wilkinson in 2016.\par
We will also use the following 
\begin{lemma} \cite{AB2012} \label{avila.bochi} For a generic $f\in\diff^{1}_{m}(M)$ if $q$ is the hyperbolic periodic point of Theorem \ref{teo.jana.acw}, for every $\eps>0$ there exists a $C^{1}$-neighborhood $\U(f)$ of $f$ such that for all $C^{2}$ diffeomorphisms $g$ in $\U(f)$:
 $$m(\phc_{g}(q_{g}))>1-\eps$$
 where the hyperbolic periodic point $q_{g}$ is the analytic continuation of $q$. 
\end{lemma}
\begin{proof} It follows from Lemma 5.1. in \cite{AB2012} and the fact that generically in $\diff^{1}_{m}(M)$ all points of the same unstable index are homoclinically related \cite{BC}.
\end{proof}
We will also use the following:
\begin{theorem}\cite{HHU2008,ACW2017}\label{teo.acw.ph} The generic volume-preserving partially hyperbolic diffeomorphism is stably Bernoulli.
\end{theorem}

\section{Stable Bernoulliness criterion} \label{section.minimality}
The following theorem provides a criterion to identify stably Bernoulli diffeomorphisms. We emphasize that this is not a result that requires $f$ to have generic properties, it holds for every diffeomorphism satisfying the hypothesis. This criterion is valid for compact manifolds of any dimension. 

\begin{theorem}\label{teo.no.generico} {\em (SB-criterion)} Let $f\in\diff^{1}_{m}(M)$ and $W$ an $f$-invariant expanding minimal foliation such that:
\begin{enumerate}
 \item there exists a $Df$-invariant sub-bundle of $TM$, $F$ such that the splitting $TM=F\oplus_< TW$ is dominated. 
 \item there exist a hyperbolic periodic point $p_{f}$ with unstable index $u(p_{f})=\dim TW$ and a $C^{1}$-neighborhood $\U(f)\subset\diff^{1}_{m}(M)$ such that for all $g\in \U(f)\cap \diff^{2}_{m}(M)$ $m(\phc^{-}(p_{g}))>0$, where $p_{g}$ is the analytic continuation of $p_{f}$.
\end{enumerate}
Then $f$ is stably Bernoulli. 
\end{theorem}
We denote by: 
$$\Gamma(f)=M \backslash \phc^u(p)$$

Here $\Gamma(f)$ is a compact, $f$-invariant and $u$-saturated subset of $M$. Also, if $f\in\diff^1_m(M)$ is ergodic then $m(\Gamma(f))=0$. For all $g \in \U(f)$ we define:
$$\Gamma(g):=M \backslash \phc^u_{g}(p_g)$$

\begin{proof}[Proof of Theorem \ref{teo.no.generico}]
Let us first see the next lemma.

\begin{lemma}\label{up.continuity} With the Hausdorff topology, the function $f \longmapsto \Gamma(f)$ is upper-semicontinuous, that is: if $f_n \stackrel{C^1}{\rightarrow} f$ then $\limsup_{n \to \infty} \Gamma(f_n) \subset \Gamma(f)$.\par
As a corollary, if $\Gamma(f)=\emptyset$, there is a $C^{1}$-neighborhood $\U$ of $f$ such that $\Gamma(g)=\emptyset$ for all $g\in\U$.
\end{lemma}

\begin{proof}[Proof of Lemma]  By considering subsequences, it is enough to see that if $x_{n}\in\Gamma(f_{n})$ and $x_{n}\to x_{*}$, then $x_{*}\in\Gamma (f)$. If this were not the case, $W^{u}(x_{*})\transv W^{s}(o(p))\ne\emptyset$. This implies that there exist open sets $U\subset M$ and $\U(f)\subset \diff^{1}_{m}(M)$ such that $W^{u}_{g}(x)\transv W^{s}_{g}(p_{g})\ne\emptyset$ for all $x\in U$ and $g\in\U(f)$. This proves the lemma. 
\end{proof}
To prove that $f$ is stably Bernoulli, note that by hypothesis $W^u$ is an $f$-invariant expanding minimal foliation, therefore $\Gamma(f)= \emptyset$. This implies, by the upper-semicontinuity, that $\Gamma(g)= \emptyset$ in a $C^{1}$-neighborhood, which we still call $\U(f)$.\\
By definition of $\Gamma(g)$, $M=\phc^u(p_g)\zeroeq\phc^{+}(p_{g})$. Now, by hypothesis $m(\phc^{-}(p_{g}))>0$. Then, by \cite{HHTU11}, $m(\phc^{u}(p_{g})\cap \phc^{-}(p_{g}))=1$ and $g$ is ergodic. Now, minimality of $W^{u}_{f}$ implies that all points in $o_{f}(p)$ are homoclinically related. This persists in an open neighborhood of $f$. So, by Corollary \ref{corollary.homoclinically.related} the result follows.

\end{proof}
\section{Proof of Theorem \ref{teo.stable.bernoulli}} 
Let $\mathcal{R} \subset \diff^{1}_{m}(M)$ a residual set such that both Theorem \ref{teo.jana.acw} and Lemma \ref{avila.bochi} apply. 
Let $f\in \mathcal{R}$ be such that there exists an invariant expanding minimal foliation $W$. Let $TM=E^{-}\oplus E^{+}$ be the zipped Oseledets splitting, as in Theorem \ref{teo.jana.acw}. We divide the proof in cases.

\begin{case}\label{case1} $\dim E^{+}=2$, and the finest Oseledets splitting is $TM=E^{-}\oplus E^{+}$\\
We claim that generically in this case $TW=E^{+}$, whence the diffeomorphism is Anosov. Indeed, by \cite[Theorem C]{DPU1999}, generically either $f$ is partially hyperbolic, or there exists a hyperbolic periodic point $p$ with unstable index $u(p)=2$ and complex eigenvalues. Since the finest Oseledets splitting is $TM=E^{-}\oplus E^{+}$, $f$ is not partially hyperbolic (otherwise, this would produce three different Lyapunov exponents), so there exists a hyperbolic periodic point with complex eigenvalues. 
This precludes the existence of a one-dimensional expanding invariant foliation $W$, so the only possible case is that $\dim W=2$, therefore $f$ is Anosov and hence stably Bernoulli.\par
We remark that the work of \cite{DPU1999} is for robustly transitive diffeomorphisms, not for volume preserving diffeomorphisms. However, their perturbation techniques can be made volume preserving, so their result holds also in this setting. 
\end{case}
\begin{case} $\dim E^{+}=2$ and the Oseledets splitting is of the form $TM=E^{-}\oplus_{<} TW\oplus_{<} E_{1}^{+}$\\
 In this case $f$ is also Anosov and the result follows.
\end{case}
\begin{case} $\dim E^{+}=2$ and the Oseledets splitting is of the form $TM=E^{-}\oplus_{<}E^{+}_{2}\oplus_{<}TW$\\
By Theorem \ref{teo.jana.acw}, there exists a hyperbolic periodic point $q$ so that $\phc(q)\zeroeq M$. Since $f$ is volume preserving and the extreme bundles are one-dimensional, $f$ is partially hyperbolic. Since $f$ is generic, Theorem \ref{teo.acw.ph} implies that $f$ is stably ergodic. Lemma \ref{avila.bochi} implies that $\phc_{g}(q_{g})\zeroeq M$ for all $C^{2}$ diffeomorphisms $g$ in a $C^{1}$-neighborhood of $f$. Now Theorem \ref{teo.bonatti.crovisier} implies that all points in $o(q)$ are homoclinically related. This persists in a $C^{1}$-neighborhood of $f$, so Corollary \ref{corollary.homoclinically.related} implies that $f$ is actually stably Bernoulli.
\end{case}
\begin{case} $\dim E^{+}=1$. Then $E^{+}=TW$\\
Lemma \ref{avila.bochi} implies we are in the hypothesis of the SB-criterion (Theorem \ref{teo.no.generico}) and the result follows. 
\end{case}

\subsection*{Acknowledgements:} We want to thank R. Ures for useful comments and discussions that helped improve this paper .

\bibliographystyle{alpha}
\bibliography{2018minimalitytesis}

\end{document}